\documentclass[11pt,reqno]{amsart}
\usepackage{graphicx}
\usepackage{verbatim}
\usepackage{textcomp}
\usepackage{amssymb}
\usepackage{cite}
\usepackage{amsmath}
\usepackage{latexsym}
\usepackage{amscd}
\usepackage{amsthm}
\usepackage{mathrsfs}
\usepackage{xypic}
\usepackage{bm}
\usepackage{url}
\usepackage{hyperref}

\newtheorem{thm}{Theorem}
\newtheorem{corr}[thm]{Corollary}
\newtheorem{lem}[thm]{Lemma}
\newtheorem{prop}[thm]{Proposition}

\theoremstyle{definition}

\newtheorem*{ack}{Acknowledgment}
\newtheorem{rem}[thm]{Remark}
\def\R{\mathbb R}

\def\H{\mathbb H}
\def\SS{\mathbb S}

\def\pt{\partial}

\begin{document}
\title[A gap theorem in hyperbolic space and hemisphere]{A gap theorem for free boundary minimal surfaces in geodesic balls of hyperbolic space and hemisphere}
\author{Haizhong Li}
\address{Department of mathematical sciences, Tsinghua University, 100084, Beijing, P. R. China}
\email{\href{mailto:hli@math.tsinghua.edu.cn}{hli@math.tsinghua.edu.cn}}
\author{Changwei Xiong}
\address{Mathematical Sciences Institute, Australian National University, Canberra, ACT 2601, Australia}
\email{\href{mailto:changwei.xiong@anu.edu.au}{changwei.xiong@anu.edu.au}}
\date{\today}
\thanks{}

\maketitle

\begin{abstract}
In this paper we provide a pinching condition for the characterization of the totally geodesic disk and the rotational annulus among minimal surfaces with free boundary in geodesic balls of three-dimensional hyperbolic space and hemisphere. The pinching condition involves the length of the second fundamental form, the support function of the surface, and a natural potential function in hyperbolic space and hemisphere.
\end{abstract}

\section{Introduction}

To characterize simple but important geometric objects by various standards has been a focus of attention for a long time. For example, F. J. Almgren \cite{Alm66} proved that the equator is the only  minimal surface in $\SS^3$ of genus $0$ (up to rigid motions in $\SS^3$). The Lawson conjecture confirmed by S. Brendle \cite{Bre13} states that the Clifford torus is the only embedded minimal torus in $\SS^3$. The Pinkall-Sterling conjecture proved by B. Andrews and H. Li \cite{AL15} says that the rotational torus is the only embedded torus with constant mean curvature in $\SS^3$. The Willmore conjecture due to Marques and Neves \cite{MN14} shows that the Clifford torus is the unique minimizer of the Willmore energy among surfaces in $\SS^3$ of genus at least $1$. In another direction, the equator and the Clifford tori $C_{m,n-m}:=S^m(\sqrt{\frac{m}{n}})\times S^{n-m}(\sqrt{\frac{n-m}{n}})$, $1\leq m\leq n-1$,  in $\SS^{n+1}$ can be characterized by the following gap theorem:

\begin{thm}[Chern-do~Carmo-Kobayashi \cite{CdK70}, Lawson \cite{Law69}, Simons \cite{Sim68} ]\label{thmA}
Let $\Sigma^n$ be a closed minimal hypersurface in the unit sphere $\SS^{n+1}$. Asssume that its second fundamental form $A$ satisfies
\begin{equation}
|A|^2\leq n.
\end{equation}
Then
\begin{enumerate}
  \item either $|A|^2\equiv 0$ and $\Sigma^n$ is an equator;
  \item or $|A|^2\equiv n$ and $\Sigma^n$ is one of Clifford tori $C_{m,n-m}$, $1\leq m\leq n-1$.
  \end{enumerate}
\end{thm}

In the setting where a minimal surface lies in a $3$-dimensional Euclidean unit ball $B^3$ with free boundary, the flat equatorial disk and the critical catenoid play analogous roles as the equator and the Clifford torus in $\SS^3$. Here the critical catenoid is a piece of a catenoid in $\R^3$ which intersects $\pt B^3$ orthogonally (free boundary). Due to their simple topology, these two examples in $B^3$ have got extensively studied these years. We refer the readers to \cite{AN16} for a nice account of the works concerning the flat equatorial disk and the critical catenoid. Or see e.g. \cite{Bre12,FL14,FS16,FS15,FS11,McG16,Nit85,RV95}. In particular, very recently Lucas~Ambrozio and Ivaldo~Nunes have obtained a gap theorem which is similar to Theorem \ref{thmA} in some sense.

\begin{thm}[Ambrozio-Nunes \cite{AN16}]\label{thmB}
Let $\Sigma$ be a compact free boundary minimal surface in Euclidean unit ball $B^3$. Assume that for all points $x\in\Sigma$,
\begin{equation}
|A|^2(x)\langle x,N(x)\rangle^2\leq 2,
\end{equation}
where $N(x)$ denotes a unit normal vector at the point $x\in \Sigma$ and $A$ denotes the second fundamental form of $\Sigma$. Then
\begin{enumerate}
  \item either $|A|^2\langle x,N(x)\rangle^2\equiv 0$ and $\Sigma$ is a flat equatorial disk;
  \item or $|A|^2(p)\langle p,N(p)\rangle^2= 2$  at some point $p\in \Sigma$ and $\Sigma$ is a critical catenoid.
\end{enumerate}
\end{thm}

In this paper we aim at proving a similar gap theorem in the $3$-dimensional hyperbolic space $\H^3$ and hemisphere $\SS^3_+$. To that end we first introduce some notations.

Let $M^3$ be a $3$-dimensional warped product Riemannian manifold $[0,R_\infty)\times \SS^2$ equipped with the metric (see \cite{LWX14})
\begin{equation}
g=dr^2+\lambda(r)^2 g_{\SS^2},
\end{equation}
where $\lambda(r)$ and $R_\infty$ are given by
\begin{equation*}
\lambda(r)=
\begin{cases}
r,& R_\infty=+\infty,\\
\sinh r,&R_\infty=+\infty,\\
\sin r,&R_\infty=\pi/2.
\end{cases}
\end{equation*}
That is, $M^3$ is Euclidean space $\R^3$, hyperbolic space $\H^3$ or hemisphere $\SS^3_+$. For any surface $\Sigma$ in $M^3$, there are two important functions, i.e. the potential function $\lambda'$ and the support function $\langle X,N\rangle$. Here $X=\lambda(r)\pt_r$ is the natural conformal vector field on $M^3$ and $N$ is a unit normal of the surface. For example, these two functions appear in the well-known Minkowski formulas.

Let $B_R=[0,R)\times \SS^2\subset M^3$ be the geodesic ball in $M^3$ with radius $R$, and $\Sigma^2$ be a compact free boundary minimal surface in $B_R$. Now we can state our main result as follows:

\begin{thm}\label{thm1}
Let $\Sigma$ be a compact free boundary minimal surface in the geodesic ball $B_R\subset M^3$ with radius $R<R_\infty$. Assume that for all points $x$ in $\Sigma$,
\begin{equation}\label{cond}
\dfrac{|A|^2\langle N(x),X\rangle^2}{(\lambda')^2}\leq 2,
\end{equation}
where $N(x)$ denotes a unit normal vector at the point $x\in \Sigma$ and $A$ denotes the second fundamental form of $\Sigma$. Then
\begin{enumerate}
  \item either $|A|^2\langle N(x),X\rangle^2\equiv 0$ and $\Sigma$ is a totally geodesic disk;
  \item or $\dfrac{|A|^2\langle N(p),X\rangle^2}{(\lambda')^2}=2$  at some point $p\in \Sigma$ and $\Sigma$ is a rotational annulus.
\end{enumerate}
\end{thm}

The proof of Theorem \ref{thm1} proceeds parallel to that of Theorem \ref{thmB}. Firstly we deal with the hyperbolic case, that is, $M^3=\H^3$. It is our key observation that the length square of the natural conformal vector $X$ can define a convex function on the surface $\Sigma$ under the pinching condition \eqref{cond}, which is the content of Lemma \ref{lem1}. Then we divide the condition \eqref{cond} into two cases: either it is a strict inequality on $\Sigma\setminus \pt\Sigma$, or it is an equality for some $p\in \Sigma\setminus \pt\Sigma$. In the latter case we can use the Minkowski space model for $\H^3$ to prove the surface is rotational, which is another difference from that in \cite{AN16}. Secondly for the spherical case $M^3=\SS^3_+$, after a little bit more work, we are able to find a convex function on $\Sigma$ under \eqref{cond}, see Lemma \ref{lem2}. Then the remaining is similar to that in hyperbolic case.

This paper is organised as follows. In Section \ref{sec2} we check that in hyperbolic case the rotational minimal annulus with free boundary in $B_R\subset \H^3$ indeed satisfies the condition \eqref{cond}, where the Poincar\'{e} disk model for $\H^3$ is convenient. Then in Section \ref{sec3} as mentioned above we prove Theorem \ref{thm1} for hyperbolic case. Finally in Section \ref{sec4} we provide the key Lemma \ref{lem2} in spherical case to finish the proof of Theorem \ref{thm1} for spherical case.


\begin{ack}
The authors would like to thank Professor Ben~Andrews for his interest in this work. The second author is also grateful to Professor Jaigyoung~Choe for discussions on minimal surfaces in hyperbolic space at the workshop on Nonlinear and Geometric Partial Differential Equations at Kioloa Campus of ANU, 2016, Australia. The first author was supported by NSFC grant No.11671224. The second author was supported by a postdoctoral fellowship funded via ARC Laureate Fellowship FL150100126.
\end{ack}

\section{The two examples for hyperbolic case}\label{sec2}

This section is devoted to examples in hyperbolic space $\H^3$ which satisfy the condition \eqref{cond}. Obviously the totally geodesic disk satisfies \eqref{cond}. While for the rotational minimal annulus with free boundary in $B_R\subset \H^3$, we verify \eqref{cond} as follows.

First we use the Minkowski space model for $\H^3\subset \R^4_1$, that is,
\begin{equation}
\H^3=\{y\in \R^4_1: y^4=\sqrt{(y^1)^2+(y^2)^2+(y^3)^2+1}\},
\end{equation}
with the metric induced from the Lorentz inner product $\langle \cdot,\cdot\rangle_1$ in $\R^4_1$. Then as in \cite{Mor81, dD83}, the rotational catenoid we consider can be parametrized as the map $F_a:\R\times \SS^1\rightarrow \H^3$ given by
\begin{align*}
F_a(s,\theta)&=\left(\sqrt{a\cosh 2s-\frac{1}{2}}\cos \theta,\sqrt{a\cosh 2s-\frac{1}{2}}\sin \theta,\right.\\
&\left.\sqrt{a\cosh 2s+\frac{1}{2}}\sinh \phi(s),\sqrt{a\cosh 2s+\frac{1}{2}}\cosh \phi(s)\right),
\end{align*}
where $\phi(s)$ is the integral
\begin{equation}
\phi(s)=\sqrt{a^2-\frac{1}{4}}\int_0^s \frac{1}{(a\cosh 2t+\frac{1}{2})\sqrt{a\cosh 2t-\frac{1}{2}}}dt,
\end{equation}
and $a>\frac{1}{2}$ is a constant. Note that $F_a$ is rotational around the plane generated by coordinate vectors $\pt_3$ and $\pt_4$, and $s$ is the arclength parameter of the generating curve.

Using this model, we can prove for any geodesic ball $B_R\subset \H^3$ centered at $(0,0,0,1)$, there exists a unique $a=a(R)$ such that $F_{a(R)}$ is a minimal annulus with free boundary in $B_R$. In fact, it suffices to prove, for any $a>\frac{1}{2}$ and fixed $\theta$ (say $0$), the tangent straight line $\{F(s)+kF'(s):k\in\R\}$ in $\R^4_1$ will intersect with the $y^4$-axis for some $s_0>0$. Then $F(s_0)$ is on the boundary of the geodesic ball. The proof is elementary, which we omit here.

However, under this model it is not easy to see that \eqref{cond} holds on $\Sigma$. To that end, we shall use the Poincar\'{e} disk model for $\H^3$. That is,
\begin{equation}
\H^3=\{z\in \R^3:\sum_{i=1}^3 (z^i)^2<1\},
\end{equation}
with the conformal flat metric
\begin{equation}
g=\rho(z)^2g_0=\frac{4}{(1-\sum_{i=1}^3(z^i)^2)^2}\sum_{i=1}^3 (dz^i)^2.
\end{equation}
Here $\rho(z)=\frac{2}{1-\sum_{i=1}^3(z^i)^2}$ and $g_0=\sum_{i=1}^3 (dz^i)^2$. Now consider the rotational minimal surface $\Sigma$ in $\H^3$ given by
\begin{equation}
z(t,\theta)=(t,f(t)\cos \theta,f(t)\sin \theta),\quad -1\leq -\bar{a}\leq t \leq \bar{a}\leq 1.
\end{equation}
As a surface in Euclidean space, $\Sigma$ has a unit normal
\begin{equation}
\bar{N}=\frac{(f',-\cos \theta,-\sin \theta)}{\sqrt{1+(f')^2}},
\end{equation}
and two principal curvatures
\begin{equation}
\bar{\kappa}_1=-\frac{f''}{(1+(f')^2)^{\frac{3}{2}}},\quad \bar{\kappa}_2=\frac{1}{f(1+(f')^2)^{\frac{1}{2}}}.
\end{equation}
Then as a surface in hyperbolic space, i.e. with a conformal metric $g=\rho^2g_0$, the surface $\Sigma$ has two principal curvatures (see e.g. \cite[Lemma 10.1.1]{Lop13})
\begin{equation}
\kappa_i=\frac{\bar{\kappa}_i}{\rho}-\frac{\bar{N}(\rho)}{\rho^2}, \quad i=1,2,
\end{equation}
here $\bar{N}(\rho)$ means the directional derivative, namely $\bar{N}(\rho)=g_0(\bar{N},\nabla_{g_0}\rho)$.

With these preparation, since $\Sigma$ is a minimal surface in $\H^3$, we can derive
\begin{equation}\label{solution}
\frac{f''}{1+(f')^2}=\frac{1}{f}+4\frac{f-tf'}{1-t^2-f^2},
\end{equation}
with initial data
\begin{equation}
f(0)=c\in (0,1),\quad f'(0)=0.
\end{equation}
Then for a geodesic ball $B_R\subset \H^3$ centered at the origin with hyperbolic radius $R$, we can find a unique $c=c(R)$ such that the solution curve $f(t)$ ($t\in [-a,a]$) generates a minimal annulus with free boundary in $B_R$. Moreover, the free boundary condition requires that $a$ is the first zero of the function $f(t)-tf'(t)$.

Now we are in a position to verify \eqref{cond}. First note that \eqref{cond} is equivalent to
\begin{equation}
\kappa_2\cdot (-<N,\pt_r>)\cdot \tanh r\leq 1,
\end{equation}
where $r$ is the geodesic distance from $z=(t,f(t)\cos \theta,f(t)\sin \theta)$ to the origin with respect to $g$. We can compute it to get
\begin{equation}
r=\ln\frac{1+|z|_{g_0}}{1-|z|_{g_0}},
\end{equation}
which implies
\begin{equation}
\tanh r=\frac{2|z|_{g_0}}{1+|z|_{g_0}^2}=\frac{2\sqrt{t^2+f^2}}{1+t^2+f^2}.
\end{equation}
Meanwhile,
\begin{align*}
<N,\pt_r>&=g_0(\bar{N},\frac{z}{|z|_{g_0}})=\frac{tf'-f}{\sqrt{1+(f')^2}\sqrt{t^2+f^2}}.
\end{align*}
In summary, we have
\begin{align}
\kappa_2&\cdot (-<N,\pt_r>)\cdot \tanh r=\left(\frac{1-t^2-f^2}{2f\sqrt{1+(f')^2}}-\frac{tf'-f}{\sqrt{1+(f')^2}}\right)\nonumber\\
&\quad \times \frac{f-tf'}{\sqrt{1+(f')^2}\sqrt{t^2+f^2}}\times \frac{2\sqrt{t^2+f^2}}{1+t^2+f^2}\nonumber\\
&=\frac{(1+f^2-t^2-2tff')(f-tf')}{f(1+(f')^2)(1+t^2+f^2)}.\label{eq1}
\end{align}
In view of the equation \eqref{solution}, we know that on $(0,a)$
\begin{equation*}
f'\geq 0,\quad f''\geq 0,\quad f-tf'>0,\text{ and }1+f^2-t^2-2tff'>0.
\end{equation*}
Furthermore, it is straightforward to check the numerator in \eqref{eq1} is decreasing, while the denominator is increasing. As a consequence, $-\kappa_2\cdot <N,\pt_r>\cdot \tanh r$ is decreasing from $1$ to $0$. This finishes the verification of \eqref{cond}.

\section{Proof of Theorem \ref{thm1} for hyperbolic case}\label{sec3}

In this section we assume $M^3=\H^3$. Firstly we prove a useful lemma.
\begin{lem}\label{lem1}
Let $\Sigma^2$ be a free boundary minimal surface in $B_R\subset \H^3$. Let $\varphi$ be the function defined by
\begin{equation}
\varphi(x)=\lambda(r(x))^2,\quad x\in \Sigma.
\end{equation}
Then we have
\begin{enumerate}
  \item[(a)] $\nabla^\Sigma \varphi(x)=2\lambda\lambda' \pt_r$ for all $x\in \pt \Sigma$.
  \item[(b)] For each $x\in \Sigma$, when $\dfrac{|A|^2\langle N(x),X\rangle^2}{(\lambda')^2}\leq 2$, the matrix $Hess_\Sigma \varphi(x)$ is positive semi-definite.
\end{enumerate}
\end{lem}
\begin{proof}
Note that $X=\lambda(r)\pt_r$ is a conformal vector field, i.e. for any vector field $v\in TM^3$ we have (see e.g. \cite{Bre13a,LWX14})
\begin{equation}
D_v(\lambda(r)\pt_r)=\lambda'(r)v.
\end{equation}
Here $D$ is the connection of ambient space $M^3$. And we will use $\nabla=\nabla^\Sigma$ to denote the connection of $\Sigma$.

Note that $\varphi(x)=|X|^2$. Let $e_1,e_2\in T_x\Sigma$ be an orthonormal basis for $T_x\Sigma$. Then for (a),
\begin{align*}
\nabla \varphi(x)&=(e_i|X|^2) e_i\\
           &=2\langle D_{e_i}X,X\rangle e_i\\
           &=2\langle \lambda' e_i,\lambda\pt_r\rangle e_i\\
           &=2\lambda\lambda' \pt_r^T,
\end{align*}
where $(\cdot)^T$ denotes the orthogonal projection onto $T_x\Sigma$. Now along $\pt\Sigma$, the vector $\pt_r$ is the conormal of $\pt\Sigma$ in $\Sigma$. So we have $\nabla \varphi(x)=2\lambda\lambda' \pt_r$ for all $x\in \pt\Sigma$, which proves (a).

For (b), taking any $Y,Z\in T_x\Sigma$, we have
\begin{align}
Hess_\Sigma& \varphi(x)(Y,Z)=YZ(|X|^2)-(\nabla_YZ)(|X|^2)\nonumber\\
                     &=2Y\langle \lambda' Z,X\rangle-2\langle \lambda' \nabla_YZ,X\rangle\nonumber\\
                     &=2\lambda'' \langle Y,\pt_r\rangle\langle Z,X\rangle+2\langle \lambda' D_YZ,X\rangle+2(\lambda')^2\langle Z,Y\rangle - 2\langle \lambda' \nabla_YZ,X\rangle\nonumber\\
                     &=2\lambda \lambda''\langle Y,\pt_r\rangle \langle Z,\pt_r\rangle+2(\lambda')^2\langle Z,Y\rangle+2\lambda'\langle A(Y),Z\rangle \langle N(x),X\rangle\nonumber\\
                     &=2\lambda \lambda''\langle Y,\pt_r\rangle \langle Z,\pt_r\rangle+2(\lambda')^2\left\langle Y+\frac{1}{\lambda'} A(Y)\langle N(x),X\rangle,Z\right\rangle.\label{eq-Hess}
\end{align}
Note that $\lambda''(r)=\sinh r\geq 0$ and the second fundamental form $A$ has eigenvalues $\pm |A|/\sqrt{2}$. Therefore when $\dfrac{|A|^2\langle N(x),X\rangle^2}{(\lambda')^2}\leq 2$, the matrix $Hess_\Sigma \varphi(x)$ is positive semi-definite.

\end{proof}

Then as in \cite{AN16}, we can obtain the following proposition. The proof for it is the same as in \cite{AN16}. We omit the details here.

\begin{prop}\label{prop1}
Let $\Sigma$ be a compact free boundary minimal surface in $B_R\subset\H^3$. Define
\begin{equation}
\mathcal{C}=\{p\in \Sigma: r(p)=\min_{x\in \Sigma} r(x)\}.
\end{equation}
If  \; $\dfrac{|A|^2\langle N(x),X\rangle^2}{(\lambda')^2}\leq 2$ on $\Sigma$, then
\begin{enumerate}
\item[(a)] either $\mathcal{C}$ contains a single point $p\in \Sigma\setminus \pt\Sigma$, in which case $\Sigma$ must be a totally geodesic disk;
\item[(b)] or $\mathcal{C}$ is a simple closed geodesic in $\Sigma\setminus \pt\Sigma$ and $\Sigma$ is homeomorphic to an annulus.
\end{enumerate}
\end{prop}
Moreover, reasoning as in \cite{AN16}, we have a corollary.
\begin{corr}
Let $\Sigma$ be a compact free boundary minimal surface in $B_R\subset\H^3$. If $\dfrac{|A|^2\langle N(x),X\rangle^2}{(\lambda')^2}< 2$ on $\Sigma$, then $\Sigma$ is a totally geodesic disk.
\end{corr}
In view of this corollary, it remains to deal with the case where the function $\dfrac{|A|^2\langle N(x),X\rangle^2}{(\lambda')^2}$ attains its maximum value $2$ at some point on $\Sigma$.
\begin{prop}
Let $\Sigma$ be a compact free boundary minimal surface in $B_R\subset\H^3$. If $\dfrac{|A|^2\langle N(x),X\rangle^2}{(\lambda')^2}\leq 2$ on $\Sigma$, and $\dfrac{|A|^2\langle N,X\rangle^2}{(\lambda')^2}= 2$ at some point $p\in \Sigma$, then $\Sigma$ is a rotational annulus.
\end{prop}

The proof is similar to that of \cite[Proposition 4]{AN16}, with some adaptation to curved space. We include it here for completeness.
\begin{proof}
The function $\varphi:\Sigma\rightarrow \R$ and $\mathcal{C}$ are defined as before. Since $\dfrac{|A|^2\langle N,X\rangle^2}{(\lambda')^2}= 2$ at some point $p\in \Sigma$, then case (a) in Proposition \ref{prop1} can not happen. In fact, should (a) occur, then by \cite[Theorem 4.1]{Sou97} or \cite[Theorem 2.1]{RS97} (an extension of Nitsche's result \cite{Nit85}), $\Sigma$ is a totally geodesic disk, a contradiction. Thus $\Sigma$ is homeomorphic to an annulus and $\mathcal{C}$ is a simple closed geodesic $\gamma: [0,l]\rightarrow \Sigma$ with arc parameter. Note that $\inf_\Sigma \varphi>0$.

Let $r_0>0$ satisfy $\sinh^2 r_0=\inf_\Sigma \varphi$ and let $S_{r_0}\subset \H^3$ be the geodesic sphere of radius ${r_0}$ centered at the origin. Then $\Sigma \subset \{x\in \H^3:r(x)\geq {r_0}\}$ and $\Sigma\cap S_{r_0}=\gamma([0,l])$. Therefore $T_{\gamma(s)}\Sigma=T_{\gamma(s)}S_{r_0}$ for all $s\in [0,l]$. Note that $\gamma$ is a geodesic in $\Sigma$. So it is also a geodesic in $S_{r_0}$. That is, $\gamma$ is a great circle of $S_{r_0}$.

Now we recall the Minkowski space model for $\H^3\subset \R^4_1$, that is,
\begin{equation}
\H^3=\{y\in \R^4_1: y^4=\sqrt{(y^1)^2+(y^2)^2+(y^3)^2+1}\},
\end{equation}
and we set $B_R\subset \H^3$ centered at $(0,0,0,1)$. Based on the analysis above, we can assume that $\pt_3$ is orthogonal to the hyperplane which contains the circle $\gamma$ and $\pt_4$. Then $\{\gamma'(s),\pt_3\}$ is an orthonormal basis of $T_{\gamma(s)}\Sigma$ for all $s\in [0,l]$.

Now as in \cite[Theorem 5.1]{Sou97}, we define for any three vectors $v_1,v_2,v_3$ in $\R^4_1$ their Lorentz cross product $v_1\wedge v_2\wedge v_3$ as the unique vector such that for any $v\in \R^4_1$:
\begin{equation}
\langle v_1\wedge v_2\wedge v_3,v\rangle_1=\det (v_1,v_2,v_3,v),
\end{equation}
where $\langle\cdot,\cdot\rangle_1$ is the Lorentz inner product and $\det$ is the determinant in the canonical basis.

With this preparation, we define a function $u:\Sigma\rightarrow \R$ by
\begin{equation}
u(x)=\langle \phi(x)\wedge \pt_3\wedge \pt_4, N(x)\rangle_1.
\end{equation}
Here $\phi:\Sigma\rightarrow \H^3\subset \R^4_1$ denotes the immersion of the surface.
It is straightforward to verify that
\begin{equation}
\Delta u+(|A|^2-2)u=0\text{ on }\Sigma,\quad \frac{\pt u}{\pt \nu}=\coth R \cdot u \text{ on }\pt\Sigma,
\end{equation}
which means that $u$ is a Jacobi function on $\Sigma$. And note that $u\equiv 0$ on $\Sigma$ if and only if $\Sigma$ is a rotational surface in $\R^4_1$ around the two-dimensional subspace generated by $\pt_3$ and $\pt_4$.

Now on the one hand, $\gamma([0,l])$ is contained in $u^{-1}(0)$, which implies at any point $\gamma(s)$ we have $\langle \nabla u,\gamma'\rangle=\frac{d(u(\gamma(s)))}{ds}=0$. On the other hand, recall $\{\gamma'(s),\pt_3\}$ is an orthonormal basis of $T_{\gamma(s)}\Sigma$ and $\gamma'(s)$ is a principal direction of $\Sigma$. Then $\pt_3$ is the other principal direction at $\gamma(s)$, which implies $D_{\pt_3} N$ is parallel to $\pt_3$. Therefore,
\begin{align*}
\langle \nabla u,\pt_3\rangle &=\langle \pt_3\wedge \pt_3\wedge \pt_4, N(x)\rangle_1+\langle \phi(x)\wedge \pt_3\wedge \pt_4, \langle \pt_3,N(x)\rangle_1 \phi(x)+D_{\pt_3}N(x)\rangle_1\\
&=0.
\end{align*}
To sum up, $\nabla u=0$ along the circle $\gamma$, which means the critical points of $u$ in the nodal set $u^{-1}(0)$ are not isolated. Then a result of S.~Y.~Cheng \cite[Theorem 2.5]{Che76} shows $u\equiv 0$ on $\Sigma$. As a consequence, $\Sigma$ is a rotational annulus as claimed.

\end{proof}


\section{Proof of Theorem \ref{thm1} for spherical case}\label{sec4}

For the problem in hemisphere, the key lemma is the following.
\begin{lem}\label{lem2}
Let $\Sigma$ be a free boundary minimal surface in $B_R\subset \SS^3$ with $R\in (0,\pi/2)$. Let $\varphi$ be the function defined by
\begin{equation}
\varphi(x)=\lambda(r(x))^2,\quad x\in \Sigma,
\end{equation}
and choose $\Phi(s)=1-\sqrt{1-s}.$
Then we have
\begin{enumerate}
  \item[(a)] $\nabla^\Sigma \Phi(\varphi)(x)=2\Phi'(\varphi)\lambda\lambda' \pt_r$ for all $x\in \pt \Sigma$.
  \item[(b)] For each $x\in \Sigma$, when $\dfrac{|A|^2\langle N(x),X\rangle^2}{(\lambda')^2}\leq 2$, the matrix $Hess_\Sigma \Phi(\varphi)(x)$ is positive semi-definite.
\end{enumerate}
\end{lem}
\begin{proof}
(a) Since $\nabla^\Sigma \Phi(\varphi)=\Phi'(\varphi)\nabla^\Sigma \varphi$, the conclusion follows immediately.

(b) First we have for any $Y,Z\in T_x\Sigma$,
\begin{align*}
Hess_\Sigma \Phi(\varphi)(Y,Z)&=YZ(\Phi(\varphi))-\nabla_YZ(\Phi(\varphi))\\
                              &=Y(\Phi'(\varphi)Z\varphi)-\Phi'(\varphi)\nabla_YZ\varphi\\
                              &=\Phi''(\varphi)Y\varphi\cdot Z\varphi+\Phi'(\varphi)YZ\varphi-\Phi'(\varphi)\nabla_YZ\varphi\\
                              &=\Phi''(\varphi)4\lambda^2(\lambda')^2\langle Y,\pt_r\rangle \langle Z,\pt_r\rangle +\Phi'(\varphi)Hess_\Sigma \varphi(Y,Z).
\end{align*}
Now using the computation in Lemma \ref{lem1}, i.e. Equality \eqref{eq-Hess}, we get
\begin{align*}
Hess_\Sigma& \Phi(\varphi)(Y,Z)=\Phi''(\varphi)4\lambda^2(\lambda')^2\langle Y,\pt_r\rangle \langle Z,\pt_r\rangle\\
& +\Phi'(\varphi)\left(2\lambda \lambda''\langle Y,\pt_r\rangle \langle Z,\pt_r\rangle+2(\lambda')^2\left\langle Y+\frac{1}{\lambda'} A(Y)\langle N(x),X\rangle,Z\right\rangle\right)\\
&=2\left(2\Phi''(\varphi)\lambda^2(\lambda')^2+\Phi'(\varphi)\lambda \lambda''\right)\langle Y,\pt_r\rangle \langle Z,\pt_r\rangle\\
&+2\Phi'(\varphi)(\lambda')^2\left\langle Y+\frac{1}{\lambda'} A(Y)\langle N(x),X\rangle,Z\right\rangle.
\end{align*}
Note that $\lambda=\sin r$, $\lambda'=\cos r$, $\lambda''=-\sin r$, $\Phi'=\frac{1}{2}(1-s)^{-\frac{1}{2}}$ and $\Phi''=\frac{1}{4}(1-s)^{-\frac{3}{2}}$. So we can check that $2\Phi''(\varphi)\lambda^2(\lambda')^2+\Phi'(\varphi)\lambda \lambda''=0$. Then the conclusion follows.

\end{proof}

The rest of the proof of Theorem \ref{thm1} for spherical case is similar to that for hyperbolic case, with minor modification. Just note that now $\Phi(\varphi)$ is a convex function under the pinching condition \eqref{cond}. Then the remaining goes through. Here we omit the details. Therefore we complete the proof of Theorem \ref{thm1}.

\begin{rem}
The existence of the rotational free boundary minimal annulus in $B_R\subset \SS^3_+$ ($R<\pi/2$) is easier now. For example, see the deep result \cite{MNS13} by Davi Maximo, Ivaldo Nunes and Graham Smith. Moreover, consider the parametrization of $\Sigma$ as $F_a:\R\times \SS^1\rightarrow \SS^3$ given by (see \cite{dD83})
\begin{align*}
F_a(s,\theta)&=\left(\sqrt{\frac{1}{2}+a\cos 2s}\cos \theta,\sqrt{\frac{1}{2}+a\cos 2s}\sin \theta,\right.\\
&\left.\sqrt{\frac{1}{2}-a\cos 2s}\sin \phi(s),\sqrt{\frac{1}{2}-a\cos 2s}\cos \phi(s)\right),
\end{align*}
where $\phi(s)$ is the integral
\begin{equation}
\phi(s)=\sqrt{\frac{1}{4}-a^2}\int_0^s \frac{1}{(\frac{1}{2}-a\cos 2t)\sqrt{\frac{1}{2}+a\cos 2t}}dt,
\end{equation}
and $-\frac{1}{2}<a<0$ is a constant. Then after a lengthy, but direct calculation, one can check that \eqref{cond} holds for $F_a$ in some geodesic ball $B_{R(a)}\subset \SS^3$ with $R(a)<\pi/2$.
\end{rem}


\bibliographystyle{Plain}

\begin{thebibliography}{10}

\bibitem{Alm66} F.~J.~Almgren~Jr., \emph{Some interior regularity theorems for minimal surfaces and an extension of Bernstein's theorem}, Ann. of Math. (2) \textbf{84} (1966), 277--292.

\bibitem{AN16} Lucas~Ambrozio and Ivaldo~Nunes, \emph{A gap theorem for free boundary minimal surfaces in the three-ball}, arXiv:1608.05689.

\bibitem{AL15} Ben~Andrews and Haizhong~Li, \emph{Embedded constant mean curvature tori in the three-sphere}, J. Differential Geom. \textbf{99} (2015), no. 2, 169--189.

\bibitem{Bre12} S.~Brendle, \emph{A sharp bound for the area of minimal surfaces in the unit ball}, Geom. Funct. Anal. \textbf{22} (2012), no.~3, 621--626.

\bibitem{Bre13} Simon~Brendle, \emph{Embedded minimal tori in $S^3$ and the Lawson conjecture}, Acta Math. \textbf{211} (2013), no.~2, 177--190.

\bibitem{Bre13a} Simon~Brendle, \emph{Constant mean curvature surfaces in warped product manifolds}, Publ. Math. Inst. Hautes \'{E}tudes Sci. \textbf{117} (2013), 247--269.

\bibitem{Che76} Shiu~Yuen~Cheng, \emph{Eigenfunctions and nodal sets}, Comment. Math. Helv. \textbf{51} (1976), no.~1, 43--55.

\bibitem{CdK70} S.~S.~Chern, M.~do~Carmo and S.~Kobayashi, \emph{Minimal submanifolds of a sphere with second fundamental form of constant length}, 1970 Functional Analysis and Related Fields (Proc. Conf. for M. Stone, Univ. Chicago, Chicago, Ill., 1968) pp. 59--75, Springer, New York.

\bibitem{dD83} M.~do~Carmo and M.~Dajczer, \emph{Rotation hypersurfaces in spaces of constant curvature}, Trans. Amer. Math. Soc. \textbf{277} (1983), no.~2, 685--709.

\bibitem{FL14} A.~Fraser and M.~Li, \emph{Compactness of the space of embedded minimal surfaces with free boundary in three-manifolds with nonnegative Ricci curvature and convex boundary}, J. Differential Geom. \textbf{96} (2014), no. 2, 183--200.

\bibitem{FS11} A.~Fraser and R.~Schoen, \emph{The first Steklov eigenvalue, conformal geometry, and minimal surfaces}, Adv. Math. \textbf{226} (2011), no. 5, 4011--4030.

\bibitem{FS15} A.~Fraser and R.~Schoen, \emph{Uniqueness theorems for free boundary minimal disks in space forms}, Int. Math. Res. Not. IMRN 2015, no. 17, 8268--8274.

\bibitem{FS16} A.~Fraser and R.~Schoen, \emph{Sharp eigenvalue bounds and minimal surfaces in the ball}, Invent. Math. \textbf{203} (2016), no. 3, 823--890.

\bibitem{Law69} H. Blaine~Lawson Jr., \emph{Local rigidity theorems for minimal hypersurfaces}, Ann. of Math. (2) \textbf{89} (1969), 187--197.

\bibitem{LWX14} Haizhong~Li, Yong~Wei and Changwei~Xiong, \emph{A note on Weingarten hypersurfaces in the warped product manifold}, Internat. J. Math. \textbf{25} (2014), no. 14, 1450121, 13 pp.

\bibitem{Lop13} Rafael~L\'{o}pez, \emph{Constant mean curvature surfaces with boundary}, Springer Monographs in Mathematics, Springer, Heidelberg, 2013.

\bibitem{MN14} Fernando C. Marques and Andr\'{e} Neves,  \emph{Min-max theory and the Willmore conjecture}, Ann. of Math. (2) \textbf{179} (2014), no. 2, 683--782.

\bibitem{MNS13} Davi~Maximo, Ivaldo~Nunes and Graham~Smith, \emph{Free boundary minimal annuli in convex three-manifolds}, to appear in J. Differential Geom.

\bibitem{McG16} P.~McGrath, \emph{A characterization of the critical catenoid}, arXiv:1603.04114.

\bibitem{Mor81} Hiroshi~Mori, \emph{Minimal surfaces of revolution in $H^3$ and their global stability},
Indiana Univ. Math. J. \textbf{30} (1981), no.~5, 787--794.

\bibitem{Nit85} Johannes~C.~C.~Nitsche, \emph{Stationary partitioning of convex bodies}, Arch. Rational Mech. Anal. \textbf{89} (1985), no.~1, 1--19.

\bibitem{RS97} Antonio~Ros and Rabah~Souam, \emph{On stability of capillary surfaces in a ball}, Pacific J. Math. \textbf{178} (1997), no. 2, 345--361.

\bibitem{RV95} A.~Ros and E.~Vergasta, \emph{Stability for hypersurfaces of constant mean curvature with free boundary}, Geom. Dedicata \textbf{56} (1995), no. 1, 19--33.

\bibitem{Sim68} James~Simons, \emph{Minimal varieties in riemannian manifolds}, Ann. of Math. (2) \textbf{88} (1968), 62--105.

\bibitem{Sou97} Rabah~Souam, \emph{On stability of stationary hypersurfaces for the partitioning problem for balls in space forms}, Math. Z. \textbf{224} (1997), no.~2, 195--208.


\end{thebibliography}

\end{document}